\documentclass[11pt,reqno]{amsart}
\pdfoutput=1
\usepackage{graphicx}
\usepackage{fixltx2e}
\usepackage[utf8]{inputenc}
\usepackage[british]{babel}
\usepackage{stix}
\usepackage{datetime}
\usepackage{mathtools}
\usepackage{amsthm}
\usepackage{enumitem}
\usepackage[protrusion=true,expansion=true]{microtype}
\usepackage[all]{nowidow}
\usepackage{tikz}
\usetikzlibrary{calc}
\newtheoremstyle{erdfn}% name
  {}%      Space above, empty = `usual value'
  {}%      Space below
  {\itshape}% Body font
  {}%         Indent amount (empty = no indent, \parindent = para indent)
  {\sffamily\fontseries{bx}\selectfont}% Thm head font
  { -- }%        Punctuation after thm head
  { }%     Space after thm head: " " = normal interword space;
  {}% Thm head spec
\newtheoremstyle{erthm}% name
  {}%      Space above, empty = `usual value'
  {}%      Space below
  {\itshape}% Body font
  {}%         Indent amount (empty = no indent, \parindent = para indent)}
  {\sffamily\fontseries{bx}\selectfont}% Thm head font
  { -- }%        Punctuation after thm head
  { }%     Space after thm head: " " = normal interword space;
  {}% Thm head spec
\newtheoremstyle{errem}% name
  {}%      Space above, empty = `usual value'
  {}%      Space below
  {}% Body font
  {}%         Indent amount (empty = no indent, \parindent = para indent)
  {\ttfamily\itshape}% Thm head font
  { -- }%        Punctuation after thm head�
  { }%     Space after thm head: " " = normal interword space;
  {}% Thm head spec
\theoremstyle{erthm}
\newtheorem{thm}{Theorem}
\newtheorem{prop}[thm]{Proposition}
\newtheorem{lem}[thm]{Lemma}

\theoremstyle{errem}
\newtheorem{rem}[thm]{Remark}
\theoremstyle{erdfn}

\mathtoolsset{
showonlyrefs,
mathic
}
\makeatletter
 \newlength{\h@uteurnumerateur}
 \newlength{\h@uteurdenominateur}
\newcommand*{\quotientdroite}[2]{%...................................................................quotient.droite
  \mathchoice%
  {%displaystyle
    \settoheight{\h@uteurnumerateur}{\ensuremath{\displaystyle{#1#2}}}%
    \settoheight{\h@uteurdenominateur}{\ensuremath{\displaystyle{#1#2}}}%
    \raisebox{0.5\h@uteurnumerateur}{\ensuremath{\displaystyle{#1}}}%
    \mkern-5mu\diagup\mkern-4mu%
    \raisebox{-0.5\h@uteurdenominateur}{\ensuremath{\displaystyle{#2}}}%
  }
  {%textstyle
    \settoheight{\h@uteurnumerateur}{\ensuremath{\textstyle{#1#2}}}%
    \settoheight{\h@uteurdenominateur}{\ensuremath{\textstyle{#1#2}}}%
  \raisebox{0.2\h@uteurnumerateur}{\ensuremath{\textstyle{#1}}}%
  /%
  \raisebox{-0.2\h@uteurdenominateur}{\ensuremath{\textstyle{#2}}}%
  }
  {%scriptstyle
    \settoheight{\h@uteurnumerateur}{\ensuremath{\scriptstyle{#1#2}}}%
    \settoheight{\h@uteurdenominateur}{\ensuremath{\scriptstyle{#1#2}}}%
  \raisebox{0.2\h@uteurnumerateur}{\ensuremath{\scriptstyle{#1}}}%
  /%
  \raisebox{-0.2\h@uteurdenominateur}{\ensuremath{\scriptstyle{#2}}}%
  }
  {%scriptscriptstyle
    \settoheight{\h@uteurnumerateur}{\ensuremath{\scriptscriptstyle{#1#2}}}%
    \settoheight{\h@uteurdenominateur}{\ensuremath{\scriptscriptstyle{#1#2}}}%
  \raisebox{0.2\h@uteurnumerateur}{\ensuremath{\scriptscriptstyle{#1}}}%
  /%
  \raisebox{-0.2\h@uteurdenominateur}{\ensuremath{\scriptscriptstyle{#2}}}%
  }
}
\newcommand*{\quotientgauche}[2]{%...................................................................quotient.gauche
  \mathchoice%
  {%displaystyle
    \settoheight{\h@uteurnumerateur}{\ensuremath{\displaystyle{#1#2}}}%
    \settoheight{\h@uteurdenominateur}{\ensuremath{\displaystyle{#1#2}}}%
    \raisebox{-0.5\h@uteurnumerateur}{\ensuremath{\displaystyle{#1}}}%
    \mkern-3mu\diagdown\mkern-5mu%
    \raisebox{0.5\h@uteurdenominateur}{\ensuremath{\displaystyle{#2}}}%
  }
  {%textstyle
    \settoheight{\h@uteurnumerateur}{\ensuremath{\textstyle{#1#2}}}%
    \settoheight{\h@uteurdenominateur}{\ensuremath{\textstyle{#1#2}}}%
  \raisebox{-0.2\h@uteurnumerateur}{\ensuremath{\textstyle{#1}}}%
  \backslash%
  \raisebox{0.2\h@uteurdenominateur}{\ensuremath{\textstyle{#2}}}%
  }
  {%scriptstyle
    \settoheight{\h@uteurnumerateur}{\ensuremath{\scriptstyle{#1#2}}}%
    \settoheight{\h@uteurdenominateur}{\ensuremath{\scriptstyle{#1#2}}}%
  \raisebox{-0.2\h@uteurnumerateur}{\ensuremath{\scriptstyle{#1}}}%
  \backslash%
  \raisebox{0.2\h@uteurdenominateur}{\ensuremath{\scriptstyle{#2}}}%
  }
  {%scriptscriptstyle
    \settoheight{\h@uteurnumerateur}{\ensuremath{\scriptscriptstyle{#1#2}}}%
    \settoheight{\h@uteurdenominateur}{\ensuremath{\scriptscriptstyle{#1#2}}}%
  \raisebox{-0.2\h@uteurnumerateur}{\ensuremath{\scriptscriptstyle{#1}}}%
  \backslash%
  \raisebox{0.2\h@uteurdenominateur}{\ensuremath{\scriptscriptstyle{#2}}}%
  }
}
\makeatother
\makeatletter
\let\@@pmod\pmod
\DeclareRobustCommand{\pmod}{\@ifstar\@pmods\@@pmod}
\def\@pmods#1{\mkern4mu({\operator@font mod}\mkern 6mu#1)}
\makeatother
\DeclarePairedDelimiter{\abs}{\lvert}{\rvert}
\DeclarePairedDelimiter{\Abs}{\lVert}{\rVert}
\newcommand*{\CC}{\mathbb{C}}
\newcommand*{\ca}{\mathfrak{a}}
\newcommand*{\cC}{\mathcal{C}}
\newcommand*{\cT}{\mathcal{T}}

\newcommand*{\cb}{\mathfrak{b}}
\newcommand*{\cush}{\mathfrak{S}_{\ell+1/2}}%
\newcommand*{\cusp}{S_{2\ell}}
\newcommand*{\dd}{%
  \mathop{\mathrm{d}\null}\mskip-\thinmuskip\mathord{\null}}
\renewcommand*{\epsilon}{\varepsilon}
\DeclareMathOperator{\ee}{e}
\newcommand*{\f}{\mathfrak{f}}
\newcommand*{\he}{\mathfrak{T}_{p^2}}
\newcommand*{\ic}{\mathrm{i}}
\DeclareRobustCommand{\Im}{\operatorname{Im}}

\newcommand*{\N}{\mathbb{N}}
\newcommand*{\pk}{\mathscr{H}}
\newcommand*{\R}{\mathbb{R}}
\DeclareRobustCommand{\Re}{\operatorname{Re}}
\DeclareMathOperator{\Sh}{Sh}

\newcommand*{\sN}{\mathcal{N}}
\newcommand*{\sumsqf}{\mathop{\sideset{}{^\flat}\sum}}
\newcommand*{\wh}[1]{\widehat{#1}}
\begin{document}
%..............................
\title[Sign of Fourier coefficients of modular forms]{Sign of Fourier coefficients of modular forms of half integral weight}
%..............................
\author[Y.-K. Lau]{Yuk-kam Lau}
\address{Yuk-kam Lau. Department of Mathematics, The University of Hong Kong, Pokfulam Road, Hong Kong, Hong Kong}
\email{yklau@maths.hku.hk}
\author[E. Royer]{Emmanuel Royer}
\address{Emmanuel Royer. (1) Université Clermont Auvergne, Université Blaise Pascal, Laboratoire de Mathématiques, BP 10448, F-63000 Clermont-Ferrand, France. (2) CNRS, UMR 6620, LM, F-63178 Aubière, France}
\email{emmanuel.royer@math.univ-bpclermont.fr}
\author[J. Wu]{Jie Wu}
\address{Jie Wu. (1) School of Mathematics, Shandong University, Jinan, Shandong 250100, China.
(2) CNRS, Institut \'Elie Cartan de Lorraine, UMR 7502, Universit\'e de Lorraine, F-54506 Van\-d\oe uvre-l\`es-Nancy, France.
(3) Universit\'e de Lorraine, Institut \'Elie Cartan de Lorraine, UMR 7502, F-54506 Van\-d\oe uvre-l\`es-Nancy, France}
\email{jie.wu@univ-lorraine.fr}
%..............................
\date{\today -- \currenttime}
%..............................
%\thanks{This work was supported by a grant from France/Hong Kong Joint Research Scheme, Procore, sponsored by the Research Grants Council of Hong Kong and the Consulate General of France in Hong Kong \& Macau (PHC PROCORE 2013, N\(^\circ\) 28212PE). This paper has been finished during the visit of E. Royer and J. Wu at Hong Kong University in 2014. They would like to thank the department of mathematics for hospitality and excellent working conditions. Finally, we express our gratitude to the anonymous referee for his insightful advice that led to a better version of Theorem~\ref{thm2}}
\thanks{This work was supported by a grant from France/Hong Kong Joint Research Scheme, Procore, sponsored by the Research Grants Council of Hong Kong (F-HK026/12T) and the Consulate General of France in Hong Kong \& Macau (PHC PROCORE 2013, N\(^\circ\) 28212PE). Lau is also supported by GRF 17302514 of the Research Grants Council of Hong Kong.}

%..............................
%..............................
%..............................
\begin{abstract}
We establish lower bounds for (i) the numbers of positive and negative terms and (ii) the number of sign changes  in the sequence of Fourier coefficients at squarefree integers of a half-integral weight modular Hecke eigenform. %
\end{abstract}
%..............................
\ddmmyyyydate
\maketitle
%..............................
\section{Introduction}
\subsection{Results}
Let \(\ell\geq 4\)  be a positive integer. Denote by \(\cush\) the vector space of all cusp forms of weight \(\ell+1/2\) for the congruence subgroup \(\Gamma_0(4)\) . The Fourier expansion of \(\f\in\cush\) at \(\infty\)  can be written as
\begin{equation}\label{eq_newone}%
\f(z)=\sum_{n=1}^{\infty}\lambda_{\f}(n) n^{\ell/2-1/4}\ee(nz) \quad (z\in\pk),
\end{equation}
where \(\ee(z) = \ee^{2\pi\ic z}\) and \(\pk\) is the Poincar\'e upper half plane. For any squarefree integer \(t\) Waldspurger \cite{MR646366} proved the following elegant formula
\begin{equation}\label{WaldspurgerFormula}
\lambda_{\f}(t)^2=C_\f L(\tfrac{1}{2}, \Sh_t\f, \chi_t),
\end{equation}
where \(\Sh_t\f\)  is the Shimura lift of \(\f\) associated to \(t\) (this is a cusp form of weight \(2\ell\) and of level 2), \(\chi_t(n)\) is a real character modulo \(t\) (defined in Section~\ref{sec_background}) and \(C_\f\) is a constant depending on \(\f\) only. In the following, the letter \(t\) will always be a squarefree integer and \(\sumsqf\) a sum over squarefree integers.

In view of \eqref{WaldspurgerFormula}, Kohnen \cite{MR2726580} posed the following question: in the case where \(\lambda_{\f}(t)\) is a real number, what is its sign? Very recently, Hulse, Kairal, Kuan \& Lim made a significant progress toward this question by proving that \(\lambda_{\f}(t)\) changes sign infinitely often if \(\f\in\cush\) is an eigenform of all the Hecke operators (see \cite[Theorem 1.1]{MR2904928}). 

In order to describe the order of magnitude of \(\lambda_{\f}(t)\), we choose \(\alpha\) a non negative real number such that the inequality
\begin{equation}\label{defvartheta}
\lambda_{\f}(t)\ll_{\f, \alpha} t^{\alpha}
\end{equation}
holds for all squarefree integers \(t\). The implied constant depends on \(\f\) and \(\alpha\) only. It is conjectured that one can take
\begin{equation}\label{RamanujanConjecture}
\alpha=\epsilon
\end{equation}
for any \(\epsilon>0\). This could be regarded as an analogue of the Ramanujan conjecture on cusp forms of integral weight. Conrey \& Iwaniec \cite[Corollary 1.3]{MR1779567} proved that one can take %
\begin{equation}\label{BoundConreyIwaniec}
\alpha=\tfrac{1}{6}+\epsilon %
\end{equation}
for any \(\epsilon>0\).

The main aim of this paper is to establish a quantitative version of the result of Hulse, Kairal, Kuan \& Lim. Define
\begin{align}\label{defTfpmx}
\cT_{\f}^{+}(x)&=\#\left\{\text{\(t\leq x\), \(t\) squarefree}\colon \lambda_{\f}(t)>0\right\}\\
\shortintertext{and}
\cT_{\f}^{-}(x)&=\#\left\{\text{\(t\leq x\), \(t\) squarefree}\colon \lambda_{\f}(t)<0\right\}.
\end{align}

We establish the following results.

\begin{thm}\label{thm1}
Let \(\ell\geq 4\) be a positive integer and \(\f\in\cush\) an eigenform of all the Hecke operators 
such that the $\lambda_{\f}(n)$ are real for all $n\ge 1$. 
Then for any \(\epsilon>0\), we have
\begin{equation}\label{thm1Equ1}
\cT_{\f}^{+}(x) \geq x^{1-2\alpha-\epsilon}, 
\qquad
\cT_{\f}^{-}(x) \geq x^{1-2\alpha-\epsilon}
\end{equation}
for all \(x\geq x_0({\mathfrak f}, \varepsilon)\), where \(\alpha\)  is given by \eqref{defvartheta} and \( x_0({\mathfrak f}, \varepsilon)\) is a positive real number depending only on \(\f\)  and \(\epsilon\). 
\end{thm}
\begin{rem}
In particular, the Conrey \& Iwaniec bound leads to
\[
\cT_{\f}^{+}(x) \geq x^{2/3-\epsilon}, 
\qquad
\cT_{\f}^{-}(x) \geq x^{2/3-\epsilon}
\]
for all \(x\geq x_0(\f, \epsilon)\).
\end{rem}
\begin{rem}
The study about the sign equidistribution of the sequence \(\left(\lambda_{\f}(tn^2)\right)_{n\in \N}\) was investigated in \cite{MR2512356}, \cite{MR2726580}, \cite{MR3010150}, \cite{MR3116654} and \cite{InamWiese2014}. In particular, Inam \& Wiese proved in~\cite{MR3116654} that, if \(t\) is a \emph{fixed} squarefree integer, then%
\[%
\lim_{x\to +\infty}\frac{\#\{\text{\(p\) prime}\colon p\leq x,\, \lambda_{\f}(tp^2)>0\}}{\#\{\text{\(p\) prime}\colon p\leq x\}}=\frac{1}{2}
\]
and %
\[%
\lim_{x\to +\infty}\frac{\#\{\text{\(p\) prime}\colon p\leq x,\, \lambda_{\f}(tp^2)<0\}}{\#\{\text{\(p\) prime}\colon p\leq x\}}=\frac{1}{2}. %
\]
\end{rem}

%\begin{rem}\label{rem_kumme}
%Let \(1\leq y\leq x\) be integers. Note that if \(\cT_{\f}^{+}(x) \geq 2y\) (resp. \(\cT_{\f}^{-}(x) \geq 2y\)), then at least \(y\) of the \(2y\) squarefree integers \(t\) satisfying \(\lambda_{\f}(t)>0\) (resp. \(\lambda_{\f}(t)<0\)) are indeed larger than \(y\) and then in \(]y,x]\). %
%\end{rem}

Let us precise what we call \emph{number of squarefree sign changes} of the sequence \(\lambda_{\f}=\left(\lambda_{\f}(t)\right)_{t\geq 0}\) (where \(\lambda_{\f}(0)=0\)) restricted to squarefree indexes \(t\). From this sequence of Fourier coefficients, we build a sequence of pairs of squarefree integers \((t_n^+,t_n^-)\), that may be finite or even void, in the following way: for any integer \(n\), we have %
\[
\lambda_{\f}(t_n^+)>0,\qquad \lambda_{\f}(t_n^-)<0, %
\]
\[%
\max(t_n^+,t_n^-)<\min(t_{n+1}^+,t_{n+1}^-),
\]
and %
\(\lambda_{\f}(t)=0\) for all squarfree integer \(t\) between \(t_n^{+}\) and \(t_n^{-}\).
%\[%
%\left\{%
%\begin{array}{l}
%\text{if \(t_n^+<t_n^-\) then \(\lambda_{\f}(t_n^++1)\leq 0\) and \(\lambda_{\f}(t_n^--1)\geq 0\)}\\
%\text{if \(t_n^+>t_n^-\) then \(\lambda_{\f}(t_n^+-1)\leq 0\) and \(\lambda_{\f}(t_n^-+1)\geq 0\)}.
%\end{array}
%\right.
%\]
The number of squarefree sign changes of \(\lambda_{\f}\) is the function defined by %
\[
\cC_{\f}(x)=\#\left\{n\geq 1\colon \max(t_n^+,t_n^-)\leq x\right\}.
\]

\begin{thm}\label{thm2}
Let \(\ell\geq 4\)  be a positive integer and \(\f\in\cush\) be an eigenform of all the Hecke operators
such that the $\lambda_{\f}(n)$ are real for all $n\ge 1$. %
For any \(\epsilon>0\), the number of squarefree sign changes of \(\lambda_\f\) satisfies
\[%
\cC_\f(x)\gg_{\f,\epsilon} x^{\frac{1-4\alpha}{5}-\epsilon}
\]
for all \(x\geq x_0(\f,\eta)\), where the constant \(x_0(\f,\eta)\) and the implied constant depends on \(\f\) and \(\epsilon\).
\end{thm}
\begin{rem}
In particular, the Conrey \& Iwaniec bound leads to
\[
\cC_\f(x)\gg_{\f,\epsilon} x^{\frac{1}{15}-\epsilon}
\]
for all \(x\geq x_0(\f, \epsilon)\).
\end{rem}
%..............................
\subsection{Methods}

To prove Theorem~\ref{thm1}, we detect signs with %
\[%
\frac{\abs{\lambda_{\f}(t)}+\lambda_{\f}(t)}{2}=\begin{cases}\lambda_{\f}(t) &\text{if \(\lambda_{\f}(t)>0\)}\\
0 &\text{otherwise.}\end{cases}
\]
Bounding the Fourier coefficients with (\ref{defvartheta}), we get plainly %
\begin{equation}\label{eq_sarm}%
\sumsqf_{t\leq x}\left(\abs*{\lambda_{\f}(t)}+\lambda_{\f}(t)\right)\log\left(\frac{x}{t}\right)\ll_{\f,\alpha}\cT_{\f}^+(x)x^\alpha\log x%
\end{equation}
(recall that the letter \(t\) is for squarefree integers hence the sum is restricted to squarefree integers). %
Then we use the analytic properties of the Dirichlet series %
\[
M(\f,s)=\sumsqf_{t\leq x}\lambda_{\f}(t)t^{-s}
\qquad\text{and}\qquad 
D(\f\otimes\overline{\f},s)=\sum_{n\geq 1}\lambda_{\f}(n)^2n^{-s}
\]
in Lemma~\ref{lemMsf} and Proposition~\ref{prop_analyticsLdeux} of \S\ref{ssec_analysis} to make an auxiliary tool -- Lemma~\ref{lem_interus}. (Note that Lemma~\ref{lemMsf} is due to  \cite{MR2904928}.)  More precisely, we utilize that the Dirichlet series defining \(M(\f,s)\) and \(D(\f\otimes\overline{\f},s)\) are absolutely convergent for \(\Re s>1\). The function \(M(\f,s)\) has an analytic continuation to \(\Re s>3/4\) whereas the function \(D(\f\otimes\overline{\f},s)\) has a meromorphic continuation to \(\Re s>1/2\) with a unique pole; this pole is at \(1\) and it is simple. Thus we can easily derive Lemma~\ref{lem_interus} and then the lower bound %
\begin{equation*}%
\sumsqf_{t\leq x}\left(\abs*{\lambda_{\f}(t)}+\lambda_{\f}(t)\right)\log\left(\frac{x}{t}\right)\gg x^{1-\alpha}. %
\end{equation*}
Theorem~\ref{thm1} follows readily. 

Theorem~\ref{thm2} rests on the following delicate device of Soundararajan~\cite{Sound11}: let \(c>0\) and \(\delta>0\), then %
\begin{multline}\label{eq_Sounddevice}%
\frac{1}{2\pi\ic}\int_{c-\ic\infty}^{c+\ic\infty}\frac{(\ee^{\delta s}-1)^2}{s^2}\xi^s\dd s\\=\begin{cases}
\min\left(\log\left(\ee^{2\delta}\xi\right),\log\left(1/\xi\right)\right)&\text{if \(\ee^{-2\delta}\leq\xi\leq 1\)}\\
0 & \text{otherwise.}
\end{cases}
\end{multline}
(Thanks to the referee for suggesting this device.)
Using it with the analytic properties of \(M(\f,s)\) and \(D(\f\otimes\overline{\f},s)\), some weighted first and second moments on short intervals are evaluated. We use these moments to detect the sign changes via the positivity of %
\[%
\sum_{m\leq A}\sumsqf_{\frac{x}{m^2}<t<\frac{x+h}{m^2}}\left(\abs*{\lambda_{\f}(t)}+\epsilon_m\lambda_{\f}(t)\right)\min\left(\log\left(\frac{x+h}{tm^2}\right),\log\left(\frac{tm^2}{x}\right)\right) %
\]
for all \((\epsilon_1,\dotsc,\epsilon_A)\in\{-1,1\}^A\).

The paper is organized as follows. Section~\ref{sec_background} is devoted to the background on half-integral weight modular forms (\S\ref{ssec_half}) and the establishment of the analytic properties for the Dirichlet series we need (\S\ref{ssec_analysis}). Theorem~\ref{thm1} is proven in Section~\ref{sec_thm1}. Theorem~\ref{thm2} is proven in Section~\ref{sec_thm2}.

\subsection*{Acknowledgement} 
{We express our hearty gratitude to the anonymous referee for his/her insightful advice that led to the current much better version of Theorem~\ref{thm2} as well as the helpful comments on presentation.  The preliminary form of this paper was finished during the visit of E. Royer and J. Wu at Hong Kong University in 2014. They would like to thank the department of mathematics for hospitality and excellent working conditions. }

%..............................
\section{Background}\label{sec_background}
%..............................
\subsection{Modular forms of half-integral weight}\label{ssec_half}
%..............................
In this section, we want to recall the basic facts we need on modular forms of half-integral weight on the congruence subgroup \(\Gamma_0(4)\). All the content of this section is classical and is to be found in the main references \cite{MR0332663} and \cite{MR2020489}. It contains however the very few that the non-specialist reader will need. %

The theta function is defined on the upper half plane \(\pk\) by %
\[%
\theta(z)=1+2\sum_{n=1}^{+\infty}\ee(n^2z) %
\]
for any \(z\in\pk\). Since the \(\theta\) function does not vanish on \(\pk\), we can define the theta multiplier: for any \(\gamma\in\Gamma_0(4)\) and \(z\in\pk\), let %
\[%
j(\gamma,z)=\frac{\theta(\gamma z)}{\theta(z)}. %
\]
If \(\gamma=\bigl(\begin{smallmatrix}a & b\\ c & d\end{smallmatrix}\bigr)\), it can be shown that \(j(\gamma,z)^2=cz+d\). For any complex number \(\xi\), let \(\xi^{1/2}\) denote \(\abs{\xi}^{1/2}\ee^{\ic\arg(\xi)/2}\) where \(-\pi<\arg(\xi)\leq\pi\). The coefficient \(j(\gamma,z)/(cz+d)^{1/2}\) is called the theta multiplier. It does not depend on \(z\) and can be explicitly described in terms of \(c\) and \(d\) (see, for example, \cite[\S 2.8]{MR1474964}). 

Let \(\ell\) be a non negative integer. A modular form of weight \(\ell+1/2\) is a holomorphic function \(\f\) on \(\pk\) satisfying %
\[%
\f(\gamma z)=j(\gamma,z)^{2\ell+1}\f(z) %
\]
for all \(\gamma\in\Gamma_0(4)\) and \(z\in\pk\), and that is holomorphic at the cusps of \(\Gamma_0(4)\). If moreover \(\f\) vanishes at the cusps of \(\Gamma_0(4)\), then \(\f\) is called a cusp form of weight \(\ell+1/2\). The congruence subgroup has three cusps: \(0\), \(-1/2\) and \(\infty\). The corresponding scaling matrices are respectively %
\[%
\sigma_0=\begin{pmatrix}0 & -1/2\\ 2 & 0\end{pmatrix},\ \sigma_{-1/2}=\begin{pmatrix}1 & 0\\ -2 & 1\end{pmatrix}\quad\text{and} \quad\sigma_\infty=\begin{pmatrix}1 & 0\\ 0 & 1\end{pmatrix}.
\]
Then, if \(\f\) is a cusp form of weight \(\ell+1/2\), the following functions have a Fourier expansion vanishing at \(\infty\): %
\[%
\f\vert_{\sigma_{0}}(z)=(2z)^{-\ell-1/2}\f\left(-\frac{1}{4z}\right)%
\quad\text{and}\quad
\f\vert_{\sigma_{-1/2}}(z)=(-2z+1)^{-\ell-1/2}\f\left(-\frac{1}{2z-1}\right). %
\]
We shall write %
\begin{equation}\label{eq_newfive}%
\f(z)=\sum_{n=1}^{+\infty}\wh{\f}(n)\ee(nz)%
\end{equation}
for the Fourier expansion of \(\f\). The set \(\cush\) of modular forms of weight \(\ell+1/2\) is a finite dimensional vector space over \(\CC\). If \(\ell\leq 3\), then \(\cush=\{0\}\). In the following, we shall assume \(\ell\geq 4\).

Shimura established a correspondence between half-integral cusp forms and integral weight cusp forms on a congruence subgroup. Niwa~\cite{MR0364106} gave a more direct proof of this correspondence and lowered the level of the congruence group involved. Fix a squarefree integer \(t\). We write \(\chi_0\) for the principal character of modulus \(2\) and define a character \(\chi_t\) by %
\[%
\chi_t(n)=\chi_0(n)\left(\frac{-1}{n}\right)^\ell\left(\frac{t}{n}\right).
\]
Let \(\f\in\cush\). Then, the Dirichlet series defined by the product%
\begin{equation}\label{eq_ShiCorUn}%
L(\chi_t,s-\ell+1)\sum_{n=1}^{+\infty}\frac{\wh{\f}(tn^2)}{n^s} %
\end{equation}
is the Dirichlet series of a cusp form of integral weight \(2\ell\) over the congruence subgroup \(\Gamma_0(2)\). We denote by \(\Sh_t\f\) this cusp form and \(\cusp\) the vector space of cusp forms of weight \(2\ell\) over \(\Gamma_0(2)\). At this point, the dependence in \(t\) of \(\Sh_t\f\) is not really clear. It will become clearer after we introduce the Hecke operators. %

The Hecke operator of half-integral weight \(\ell+1/2\) and order \(p^2\) is the linear endomorphism \(\he\) on \(\cush\) that sends any cusp form with Fourier coefficients \((\wh{\f}(n))_{n\geq 1}\) to the cusp form with Fourier coefficients defined by %
\[%
\wh{\he(\f)}(n)=\wh{\f}(p^2n)+\chi_0(p)\left(\frac{(-1)^\ell n}{p}\right)p^{\ell-1}\wh{f}(n)+\chi_0(p)p^{2\ell-1}\wh{\f}\left(\frac{n}{p^2}\right).
\]
If \(n/p^2\) is not an integer, then \(\wh{\f}(n/p^2)\) is considered to be \(0\). Hecke operators and the Shimura correspondence commute, meaning that if \(T_p\) is the Hecke operator of order \(p\) over \(\cusp\), then %
\[%
\Sh_t(\he\f)=T_p(\Sh_t\f)%
\]
for any \(\f\in\cush\). In particular, if \(\f\) is an eigenform of \(\he\), then \(\Sh_t\f\) is an eigenform of \(T_p\) with same eigenvalue. Let \(\f\) be an eigenform of all the Hecke operators \(\he\): denote by \(w_p\) the corresponding eigenvalue. One has %
\begin{equation}\label{eq_ShiCorDeux}%
L(\chi_t,s-\ell+1)\sum_{n=1}^{+\infty}\frac{\wh{\f}(tn^2)}{n^s}=\wh{\f}(t)\prod_{p}\left(1-\frac{\omega_p}{p^s}+\frac{\chi_0(p)}{p^{2s-2\ell+1}}\right)^{-1} %
\end{equation}
the product being over all prime numbers. This product is the \(L\)-function of a cusp form in \(\cusp\). We denote by \(\Sh\f\) this cusp form. Remark that it does not depend on \(t\) and that \(\Sh_t\f=\wh{\f}(t)\Sh\f\). 

Let \(\psi\) be the arithmetic function defined by %
\[%
\psi(n)=\prod_{p\mid n}\left(1+p^{-1/2}\right)
\]
the product being on prime numbers. We write \(\tau\) for the divisor function and clearly \(\psi(n)\leq\tau(n)\) for every \(n\in\N^*\). The next Lemma improves slightly Lemma 4.1 in~\cite{MR2904928}. %
\begin{lem}\label{lemma6}
Let \(\f\in\cush\) be an eigenform of all the Hecke operators \(\he\). There exists a constant \(C>0\) such that, for any squarefree integer \(t\) and any integer \(n\) we have %
\[
\abs*{\wh{\f}(tn^2)}\leq C\abs*{\wh{\f}(t)}n^{\ell-1/2}\tau(n)\psi(n). %
\]
\end{lem}

\begin{proof}
From~\eqref{eq_ShiCorDeux} we get %
\begin{equation}\label{eq_uks}%
\wh{\f}(tn^2)=\wh{\f}(t)\sum_{d\mid n}\chi_t\left(\frac{n}{d}\right)\mu\left(\frac{n}{d}\right)\left(\frac{n}{d}\right)^{\ell-1}\wh{\Sh\f}(d). %
\end{equation}
By the Deligne estimate, there exists \(C>0\) such that %
\begin{equation}\label{eq_kaks}%
\abs*{\wh{\Sh\f}(d)}\leq Cd^{(2\ell-1)/2}\tau(d) %
\end{equation}
for any \(d\).  It follows from~\eqref{eq_uks} and~\eqref{eq_kaks} that %
\[%
\abs*{\wh{\f}(tn^2)}\leq C\abs*{\wh{\f}(t)}n^{\ell-1}\sum_{d\mid n}\abs*{\mu\left(\frac{n}{d}\right)}d^{1/2}\tau(d)\leq C\abs*{\wh{\f}(t)}n^{\ell-1/2}\tau(n)\psi(n). %
\]
\end{proof}
The size of the Fourier coefficients of a half integral weight modular form is therefore controlled by the size of its Fourier coefficients at squarefree integers. Deligne's bound for integral weight modular forms does not apply, although it conjecturally does. Let \(\alpha\) be a positive real number such that, if \(\f\in\cush\), then %
\[
\abs{\wh{\f}(t)}\leq Ct^{(\ell+1/2-1)/2+\alpha}%
\]
for any squarefree integer \(t\) (and \(C\) is a real number depending only on \(\f\) and \(\alpha\)). Ramanujan-Petersson conjecture asserts that \(\alpha\) can be taken arbitrarily small. The best proven result is due to Conrey \& Iwaniec~\cite{MR1779567} (see also the Appendix by Mao in~\cite{MR2296066} for an uniform value of \(C\)). Their result implies that we can take \(\alpha=1/6+\epsilon\) with any real positive \(\epsilon\).  If \(\f\in\cush\) is an eigenform of all the Hecke operators, we have by comparison of~\eqref{eq_newone} and~\eqref{eq_newfive} %
\[%
\lambda_{\f}(n)=\frac{\wh{\f}(n)}{n^{(\ell+1/2-1)/2}}.
\]
For any squarefree integer \(t\) and integer \(n\), we have then%
\begin{equation}\label{eq_bound}
\abs*{\lambda_{\f}(tn^2)}\leq C_1\abs*{\lambda_{\f}(t)}\tau(n)\psi(n)\leq C_2t^\alpha\tau(n)\psi(n) %
\end{equation}
with the admissible choice \(\alpha= 1/6+\epsilon\),  where \(C_1\) and \(C_2\) are positive real numbers not depending on \(t\) or \(n\).
%..............................
\subsection{Some associated Dirichlet series}\label{ssec_analysis}
%..............................
Let \(\f\in\cush\), and assume it is an eigenform of all the Hecke operators. We define %
\begin{equation}\label{eq_defLdeux}%
D(\f\otimes\overline{\f},s)=\sum_{n=1}^{+\infty}\lambda_{\f}(n)^2n^{-s}.
\end{equation}
Write \(\sigma=\Re s\) and \(\tau={\rm Im}\, s\).\footnote{No confusion will arise with the divisor function \(\tau(n)\) from the context.} 
According to~\eqref{eq_bound}, we know it is absolutely convergent as soon as \(\sigma >1+2\alpha\). We state analytical informations on this function. The proof is quite standard, but since we have not found a handy proof  in the literature for this case, we provide the details for completeness. %
\begin{prop}\label{prop_analyticsLdeux}
Let \(\f\in\cush\), and assume it is an eigenform of all the Hecke operators. The Dirichlet series~\eqref{eq_defLdeux} converges absolutely as soon as \(\Re s>1\). It can be continued analytically to a meromorphic function in the half plane \(\Re s>\tfrac{1}{2}\)  with the only pole at \(s=1\) . This pole is simple. Further for any \(\epsilon>0\)  we have
\begin{equation}\label{ConvexityL2fs}
D(\f\otimes\overline{\f},s)\ll_{\f,\epsilon}\abs{\tau}^{2\max(1-\sigma,0) + \epsilon}\qquad\left(\tfrac{1}{2}+\epsilon\leq \sigma\leq 3, \, \abs{\tau}\ge 1\right).
\end{equation}
The implied constant depends on \(\f\)  and \(\epsilon\)  only.
\end{prop}
\begin{proof}
Let \(\ca\) be a cusp of \(\Gamma=\Gamma_0(4)\). We denote by \(\Gamma_{\ca}\) its stability group, and by \(\sigma_{\ca}\) its scaling matrix (see~\cite[\S 2.3]{MR1474964}). The Eisenstein series associated to \(\ca\) is %
\begin{align*}%
E_{\ca}(z,s)&=\sum_{\gamma\in\quotientgauche{\Gamma_{\ca}}{\Gamma}}\Im(\sigma_{\ca}^{-1}\gamma z)^s\\
&=\sum_{\gamma\in\quotientgauche{\Gamma_\infty}{\Gamma}}\Im(\gamma \sigma_{\ca}^{-1}z)^s=E_{\infty}(\sigma_{\ca}^{-1}z,s).
\end{align*}
We take \(\{0, -1/2,\infty\}\) as a representative set of cusps and obtain %
\[%
E_0(z,s)=E_\infty\left(-\frac{1}{4z},s\right)\quad\text{and}\quad E_{-1/2}(z,s)=E_\infty\left(-\frac{z}{2z-1},s\right).
\]
These series converge absolutely for \(\Re s>1\) (see, for example~\cite[Theorem 2.1.1]{MR0429749}). Moreover, \(\f_{\vert\sigma_{\ca}}\) admits a Fourier expansion %
\[%
\f_{\vert\sigma_{\ca}}(z)=\sum_{n=1}^{+\infty}n^{(\ell+1/2-1)/2}\lambda_{\f,\ca}(n)\ee(nz). %
\]
Let %
\begin{equation}\label{eq_convolcusp}%
D(\f_\ca\otimes\overline{\f_\ca},s)=\sum_{n=1}^{+\infty}\abs*{\lambda_{\f,\ca}(n)}^2n^{-s}.
\end{equation}
Classically (see, for example, \cite[\S 13.2]{MR1474964}), we have %
%\begin{multline*}
\[%
(4\pi)^{s+\ell-1/2}\Gamma\left(s+\ell-\frac{1}{2}\right)D(\f_\ca\otimes\overline{\f_\ca},s)%\\
=\int_{\quotientgauche{\Gamma}{\pk}}y^{\ell+1/2}\abs*{\f (z)}^2E_{\ca}(z,s)\frac{\dd x\dd y}{y^2} %
\]
%\end{multline*}
for \(\Re s\) large enough. The right hand side provides an analytic continuation in the region \(\Re s>1\). By Landau Lemma, this implies that the Dirichlet series~\eqref{eq_convolcusp} is absolutely convergent for \(\Re s>1\). The general theory implies that \(s\mapsto E_{\ca}(z,s)\) has a meromorphic continuation to the whole complex plane and satisfies the functional equation %
\[%
\vec{E}(z,s)=\Phi(s)\vec{E}(z,1-s) %
\]
where \(\vec{E}\) is the transpose of \((E_\infty,E_{0},E_{-1/2})\) and \(\Phi=\left(\varphi_{\ca,\cb}\right)_{(\ca,\cb)\in\{\infty,0,-1/2\}^2}\) is the scattering matrix. Indeed, %
\[%
\varphi_{\ca,\cb}(s)=\pi^{1/2} \frac{\Gamma(s-\frac12)}{\Gamma(s)}\sum_{c>0}\sN(c)c^{-2s}%
\]
where \(\sN(c)\) is the number of \(d\), incongruent modulo \(c\) such that, there exist \(a\) and \(b\) satisfying %
\[%
\sigma_{\ca}\begin{pmatrix}a & b\\ c & d\end{pmatrix}\sigma_{\cb}^{-1}\in\Gamma_0(4).
\]
This leads to %
\begin{align*}%
\Phi(s)& =\frac{\Lambda(2s-1)}{\Lambda(2s)}\frac{2^{1-2s}}{2^{2s}-1}\begin{pmatrix}%
1 & 2^{2s-1}-1 & 2^{2s-1}-1\\
2^{2s-1}-1 & 1 & 2^{2s-1}-1\\
2^{2s-1}-1 & 2^{2s-1}-1 & 1
\end{pmatrix}\\ %
&=\frac{\Lambda(2s-1)}{\Lambda(2s)}\Psi(s), \mbox{ say}, 
\end{align*}
where  \(\Lambda(s)=\pi^{-s/2}\Gamma(s/2)\zeta(s)\). On the half plane \(\Re s\geq 1/2\), \(E_\ca\) and \(\varphi_{\ca,\ca}\) have the same poles of the same orders \cite[Theorems 4.4.2, 4.3.4, 4.3.5]{MR0429749}. The only pole on \(\Re s\geq 1/2\) is then \(s=1\) and it is simple. Note that this follows also from the general theory since we are working on a congruence subgroup (\cite[Theorem 11.3]{MR1942691}). %

Let \(\vec{L}(\f\otimes\overline{\f},s)\) be the transpose of %
\[%
\left(D(\f\otimes\overline{\f},s),D(\f_0\otimes\overline{\f_0},s),D(\f_{-1/2}\otimes\overline{\f_{-1/2}},s)\right) %
\]
and %
\[%
\vec{\Lambda}(\f,s)=(2\pi)^{-2s}\Gamma(s)\Gamma(s+\ell-1/2)\zeta(2s)\vec{L}(\f\otimes\overline{\f},s). %
\]
We proved that %
\begin{itemize}
\item \(\vec{\Lambda}(\f,s)=\Psi(s)\vec{\Lambda}(\f,1-s)\)
\item in the half plane \(\Re s\geq 1/2\), the function \(D(\f_\ca\otimes\overline{\f_\ca},s)\) has only a simple pole at \(s=1\).
\end{itemize}
Now, let \(\Abs{\cdot}\) denote the Euclidean norm in \(\R^3\). Using \(\Abs{D(\f_\ca\otimes\overline{\f_\ca},1+\varepsilon+\ic\tau)}\ll_{\f,\epsilon} 1\) for any \(\tau\in\R\) and any fixed \(\epsilon >0\), we deduce  %
\[%
\abs*{\zeta(-2\epsilon +2\ic\tau)}\cdot\Abs*{\vec{L}(\f\otimes\overline{\f},-\epsilon+\ic\tau)}\ll_{\f,\epsilon} (1+\abs{\tau})^{2+\epsilon}
\]
from the functional equation,  and the estimate
\[
\abs{\zeta(2s)}\cdot\Abs*{\vec{L}(\f\otimes\overline{\f},s)}\ll_{\f,\epsilon} (1+\abs{\tau})^{2(1-\sigma)+\epsilon}\quad (s= \sigma+\ic\tau, \ \sigma \in [0,1], \ \abs{\tau}\geq 1)
\]
by the standard argument with the convexity principle.\footnote{One needs the estimate \(\abs{\zeta(2s)}\cdot\Abs{\vec{L}(\f\otimes\overline{\f},s)}\ll\ee^{\ee^{\eta\abs{\tau}}}\) for some \(\eta>0\) in the strip so as to apply the convexity principle. This can be easily verified by the Fourier expansion of \(E_\ca(z,s)\) and \cite[(2.2.6)-(2.2.11)]{MR0429749}.}  This leads to the desired result.
\end{proof}

Another useful Dirichlet series is %
\begin{equation}\label{defMs}
M(\f,s)= \sumsqf_{t\geq 1} \lambda_{\f}(t)t^{-s}. %
\end{equation}
%where we defined %
%\[%
%\sumsqf_na(n)=\sum_{\text{\(n\) squarefree}}a(n)=\sum_n\mu(n)^2a(n). %
%\]
The series \(M(\f,s)\) is absolutely convergent for \(\Re s>1\) by the Cauchy-Schwarz inequality and Proposition~\ref{prop_analyticsLdeux}.
The next lemma is due to Hulse, Kiral, Kuan \& Lim \cite[Proposition 4.4]{MR2904928}.
\begin{lem}\label{lemMsf}
Let \(\ell\geq 4\) be a positive integer and \(\f\in\cush\) be an eigenform of all the Hecke operators. The series \(M(\f,s)\), given by \eqref{defMs}, converges for \(\Re s>\tfrac{3}{4}\). Further for any \(\epsilon>0\) we have
\begin{equation}\label{UBMs}
M(\f,\sigma+\ic\tau)\ll_{\f, \epsilon} (\abs{\tau}+1)^{\max(1-\sigma,0)+2\epsilon} \qquad (\tfrac{3}{4}+\epsilon\leq\sigma\leq 3, \abs{\tau}\geq 1)
\end{equation}
where the implied constant depends on \(\f\)  and \(\epsilon\)  only.
\end{lem}
\begin{proof}
We only sketch the proof since it is nearly the same as in \cite[Proposition 4.4]{MR2904928}. By the relation %
\[%
\mu(m)^2=\sum_{r^2\mid m}\mu(r) %
\]
we have 
\begin{equation}\label{eq_MDr}%
M(\f,s)=\sum_{r=1}^{+\infty}\mu(r)D_r(s) %
\end{equation}
where %
\[%
D_r(s)=\sum_{\mathclap{\substack{m=1\\m\equiv 0\pmod*{r^2}}}}^{+\infty}{\lambda_{\f}(m)}m^{-s}.
\]
This series is absolutely convergent for \(\Re s>1\) by Cauchy-Schwarz inequality and Proposition~\ref{prop_analyticsLdeux}. Then, introducing additive characters to remove the congruence condition and applying the Mellin transform, we get %
\[%
D_r(s)=\frac{(2\pi)^{s+(\ell+1/2-1)/2}}{\Gamma(s+(\ell+1/2-1)/2)}\cdot\frac{1}{r^2}\sum_{d\mid r^2}\sum_{\substack{u\pmod*{d}\\(u,d)=1}}\Lambda\left(\f,\frac{u}{d},s\right) %
\]
with %
\[%
\Lambda(\f,q,s)=\int_0^{+\infty}\f(\ic y+q)y^{s+(\ell-1/2)/2}\frac{\dd y}{y} %
\]
for any rational number \(q\). Using the functional equation for \(\Lambda(\f,q,s)\) (see \cite[Lemma 4.3]{MR2904928}), we obtain %
\[%
D_r(-\epsilon+\ic\tau)\ll_{\epsilon,\f}(1+\abs{\tau})^{1+2\epsilon}r^{2+5\epsilon}.
\]
From~\eqref{eq_bound}, we have also %
\[%
D_r(1+\epsilon+\ic\tau)\ll_{\epsilon,\f}\frac{1}{r^2}. %
\]
Finally, by the Phrägmen-Lindelöf principle, we deduce %
\[%
D_r(\sigma+\ic\tau)\ll_{\epsilon,\f}(1+\abs{\tau})^{1-\sigma+\epsilon}r^{2-4\sigma+\epsilon}.
\]
Reinserting this bound into~\eqref{eq_MDr} leads to the result. %
\end{proof}
%..............................
\section{Proof of Theorem~\ref{thm1}}\label{sec_thm1}
%..............................
We begin by establishing mean value results for the Fourier coefficients at squarefree integers. %
\begin{lem}\label{lem_interus}
Let \(\f\in\cush\), and assume it is an eigenform of all the Hecke operators. Let \(\epsilon>0\). There exist positive real numbers \(C_1\), \(C_2\) and \(C_3\) such that, for any \(x\geq 1\), we have %
\[%
\sumsqf_{t\leq x}\lambda_{\f}(t)\log\left(\frac{x}{t}\right)\leq C_1x^{3/4+\epsilon} %
\]
and %
\[%
C_2x\leq\sumsqf_{t\leq x}\lambda_{\f}(t)^2\leq C_3x %
\]
for any \(x\geq x_0(\f)\).
\end{lem}
\begin{proof}
Using the Perron formula~\cite[Theorem II.2.3]{MR1342300}, we write %
\[%
\sumsqf_{t\leq x}\lambda_{\f}(t)\log\left(\frac{x}{t}\right)=\frac{1}{2\pi\ic}\int_{2-\ic\infty}^{2+\ic\infty}M(\f,s)x^s\frac{\dd s}{s^2}. %
\]
We move the line of integration to \(\Re s=3/4+\epsilon\) and use Lemma~\ref{lemMsf} to have %
\[%
\sumsqf_{t\leq x}\lambda_{\f}(t)\log\left(\frac{x}{t}\right)\leq C_1x^{3/4+\epsilon}.
\]

For the second formula, we use an effective version of the Perron formula~\cite[Corollary II.2.2.1]{MR1342300}: %
\[%
\sum_{n\leq x}\lambda_{\f}(n)^2=\frac{1}{2\pi\ic}\int_{\kappa-\ic T}^{\kappa+\ic T}D(\f\otimes\overline{\f},s)x^s\frac{\dd s}{s}+O\left(\frac{x^{1+2\alpha+\epsilon}}{T}\right) %
\]
for any \(T\leq x\) and \(\kappa=1+1/\log x\). Proposition~\ref{prop_analyticsLdeux} allows to shift the line of integration to \(\Re s=1/2+\epsilon\). We get %
\[%
\frac{1}{2\pi\ic}\int_{\kappa-\ic T}^{\kappa+\ic T}D(\f\otimes\overline{\f},s)x^s\frac{\dd s}{s}=r_{\f}x+\frac{1}{2\pi\ic}\int_{\mathcal{L}}D(\f\otimes\overline{\f},s)x^s\frac{\dd s}{s}
\]
where \(r_{\f}\) is the residue at \(s=1\) of \(D(\f\otimes\overline{\f},s)\) and \(\mathcal{L}\) is the contour made from segments joining in  order the points \(\kappa-\ic T\), \(1/2+\epsilon-\ic T \), \(1/2+\epsilon+\ic T\) and \(\kappa+\ic T\). With the convexity bound in Proposition~\ref{prop_analyticsLdeux} we have %
\[%
\int_{1/2+\epsilon\pm\ic T}^{\kappa\pm\ic T}D(\f\otimes\overline{\f},s)x^s\frac{\dd s}{s}\ll\frac{x^{1+\epsilon}}{T} %
\]
if \(T\leq x^{1/2}\) and 
\[%
\int_{1/2+\epsilon-\ic T}^{1/2+\epsilon+\ic T}D(\f\otimes\overline{\f},s)x^s\frac{\dd s}{s}\ll x^{1/2+\epsilon}T. %
\]
We choose \(T=x^{1/4+\alpha}\) and obtain %
\begin{equation}\label{eq_kolm}
\sum_{n\leq x}\lambda_{\f}(n)^2=r_{\f}x+O\left(x^{3/4+\alpha+\epsilon}\right).
\end{equation}
Each positive integer \(n\) may be decomposed uniquely as \(n=tm^2\) with squarefree \(t\). Using~\eqref{eq_bound} we have %
\begin{align*}
\sum_{n\leq x}\lambda_{\f}(n)^2&\ll_{\f}\sumsqf_{t\leq x}\lambda_{\f}(t)^2\sum_{m\leq(x/t)^{1/2}}\tau(m)\psi(m)\\
&\ll_{\f}x^{1/2}\sumsqf_{t\leq x}\frac{\lambda_{\f}(t)^2}{t^{1/2}}\log\left(\frac{x}{t}\right). %
\end{align*}
Combining this with~\eqref{eq_kolm} we find %
\begin{equation}\label{eq_neli}
\sumsqf_{t\leq x}\frac{\lambda_{\f}(t)^2}{t^{1/2}}\log\left(\frac{x}{t}\right)\geq c_1x^{1/2}\qquad(x\geq x_0(\f)) %
\end{equation}
where the constant \(c_1\) depends only on \(\f\). On the other hand, \eqref{eq_kolm} leads to %
\begin{equation}\label{eq_viis}
\sumsqf_{t\leq x}\frac{\lambda_{\f}(t)^2}{t^{1/2}}\log\left(\frac{x}{t}\right)\leq\sum_{n\leq x}\frac{\lambda_{\f}(n)^2}{n^{1/2}}\log\left(\frac{x}{n}\right)\leq c_2x^{1/2} %
\end{equation}
where \(c_2\) depends only on \(\f\). Let \(c_3\in]0,1[\). From~\eqref{eq_neli} and~\eqref{eq_viis}, it follows that %
\begin{align*}
\frac{\log(1/c_3)}{(c_3x)^{1/2}}\sumsqf_{c_3x<t\leq x}\lambda_{\f}(t)^2&\geq\sumsqf_{c_3x<t\leq x}\frac{\lambda_{\f}(t)^2}{t^{1/2}}\log\left(\frac{x}{t}\right)\\
&= \sumsqf_{t\leq x}\frac{\lambda_{\f}(t)^2}{t^{1/2}}\log\left(\frac{x}{t}\right)-\sumsqf_{t\leq c_3x}\frac{\lambda_{\f}(t)^2}{t^{1/2}}\log\left(\frac{x}{t}\right)\\
&\geq\left(c_1-c_2c_3^{1/2}\right)x^{1/2}. %
\end{align*}
We deduce %
\[%
\sumsqf_{c_3x<t\leq x}\lambda_{\f}(t)^2\geq\frac{c_3^{1/2}}{\log(1/c_3)}\left(c_1-c_2c_3^{1/2}\right)x. %
\]
Choosing \(c_3<\min(1,c_1^2/c_2^2)\) we have %
\begin{equation*}\label{eq_yat}
\sumsqf_{t\leq x}\lambda_{\f}(t)^2\gg \sumsqf_{c_3x<t\leq x}\lambda_{\f}(t)^2\gg x. %
\end{equation*}

Finally, ~\eqref{eq_kolm} gives %
\[
\sumsqf_{t\leq x}\lambda_{\f}(t)^2\leq\sum_{n\leq x}\lambda_{\f}(n)^2\ll x
\]
hence
\[
\sumsqf_{t\leq x}\lambda_{\f}(t)^2\asymp x. 
\]
\end{proof}
With this Lemma, we can complete the proof of Theorem~\ref{thm1}. From~\eqref{eq_bound} we derive %
\begin{align*}
\sumsqf_{t\leq x}\abs*{\lambda_{\f}(t)}\log\left(\frac{x}{t}\right)&\gg x^{-\alpha}\sumsqf_{t\leq x}\abs*{\lambda_{\f}(t)}^2\log\left(\frac{x}{t}\right)\\
&\gg x^{-\alpha}\sumsqf_{t\leq x/2}\abs*{\lambda_{\f}(t)}^2. %
\end{align*}
Hence, Lemma~\ref{lem_interus} implies %
\begin{equation}\label{eq_kuus}
\sumsqf_{t\leq x}\abs*{\lambda_{\f}(t)}\log\left(\frac{x}{t}\right)\gg_{\f,\alpha}x^{1-\alpha}. 
\end{equation}
We detect signs of Fourier coefficients with the help of %
\[%
\frac{\abs{\lambda_{\f}(t)}+\lambda_{\f}(t)}{2}=\begin{cases}%
\lambda_{\f}(t) & \text{if \(\lambda_{\f}(t)>0\)}\\
0 & \text{otherwise}.
\end{cases}
\]
Using~\eqref{eq_bound}, we have %
\begin{equation}\label{eq_seitse}
\sumsqf_{t\leq x}\left(\abs{\lambda_{\f}(t)}+\lambda_{\f}(t)\right)\log\left(\frac{x}{t}\right)\ll\cT_{\f}^+(x)x^\alpha\log x. %
\end{equation}
Moreover, \eqref{eq_kuus} and Lemma~\ref{lem_interus} imply %
\begin{align}
\sumsqf_{t\leq x}\left(\abs{\lambda_{\f}(t)}+\lambda_{\f}(t)\right)\log\left(\frac{x}{t}\right)&=\sumsqf_{t\leq x}\abs{\lambda_{\f}(t)}\log\left(\frac{x}{t}\right)+\sumsqf_{t\leq x}\lambda_{\f}(t)\log\left(\frac{x}{t}\right)\notag\\%
&\gg x^{1-\alpha}+O\left(x^{3/4+\epsilon}\right)\notag\\
&\gg x^{1-\alpha}.\label{eq_kaheksa}
\end{align}
Finally, equations~\eqref{eq_seitse} and~\eqref{eq_kaheksa} give %
\[%
\cT_{\f}^+(x)\gg\frac{x^{1-2\alpha}}{\log x}\cdot
\]
Similarly, using %
\[%
\frac{\abs{\lambda_{\f}(t)}-\lambda_{\f}(t)}{2}=\begin{cases}%
-\lambda_{\f}(t) & \text{if \(\lambda_{\f}(t)<0\)}\\
0 & \text{otherwise}
\end{cases}
\]
we obtain %
\[%
\cT_{\f}^-(x)\gg\frac{x^{1-2\alpha}}{\log x}\cdot
\]
This finishes the proof of Theorem~\ref{thm1}.
%.............................
\section{Proof of Theorem~\ref{thm2}}\label{sec_thm2}
%.............................
The basic idea of proof is the same as for Theorem~\ref{thm1}, although here we localize on short intervals. The device~\eqref{eq_Sounddevice}  with the analytic properties of \(M(\f,s)\) gives a nice mean value estimate for \(\lambda_{\f}(t)\) over the squarefree integers in a short interval, see \eqref{eq_momun}. However our series \(D(\f\otimes\overline{\f},s)\) runs over all positive (not just squarefree) integers. We cannot obtain a counterpart for \(|\lambda_{\f}(t)|^2\). To get around, we consider a bundle of short intervals and lead to two moment estimates \eqref{eq_momentun} and \eqref{eq_momentdeux} in \S\ref{thm2p1}. Then we can enumerate the sign changes in \S\ref{thm2p2}.%

\subsection{Computation of moments of order \(1\) and \(2\)}\label{thm2p1}
%.............................
Let 
$$
0\le \alpha< 1/4
\qquad\text{and}\qquad
1>\eta>3/4+\alpha. 
$$
Suppose that \(x\) is sufficiently large. We set \(h=x^\eta\) and define \(\delta\) by \(\ee^{2\delta}=1+h/x\). We have \(\delta \asymp h/x\). %

For all \(s\in\CC\) such that \(\abs*{\Re s}\leq 2\), we have \((\ee^{\delta s}-1)^2/s^2\ll\min\left(\delta^2,1/\abs{s}^2\right)\). It follows then by Lemma~\ref{lemMsf} and~\eqref{eq_Sounddevice} that %
\begin{align}%
\makebox[1cm][l]{\(\displaystyle\sumsqf_{x\leq t\leq x+h}\lambda_{\f}(t)\min\left(\log\left(\frac{x+h}{t}\right),\log\left(\frac{t}{x}\right)\right)\)}&&\notag\\
&= \frac{1}{2\pi\ic}\int_{3/4+\epsilon-\ic\infty}^{3/4+\epsilon+\ic\infty}M(\f,s)\frac{(\ee^{\delta s}-1)^2}{s^2} x^s\dd s
\notag\\
&\ll  x^{3/4+\epsilon}\int_{-\infty}^{+\infty}\left(\abs{\tau}+1\right)^{1/4+\epsilon}\min\left(\delta^2,\frac{1}{1+\abs{\tau}^2}\right)\dd\tau
\notag\\
&\ll  h^{3/4}x^\epsilon.\label{eq_momun} %
\end{align}
For any integer constant \(A>0\), let \((\epsilon_1,\dotsc,\epsilon_A)\in\{-1,1\}^A\). The bound for the moment of order \(1\) follows from~\eqref{eq_momun}, that is %
\begin{equation}\label{eq_momentun}
\sum_{m\leq A}\epsilon_m\sumsqf_{\frac{x}{m^2}<t<\frac{x+h}{m^2}}\lambda_{\f}(t)\min\left(\log\left(\frac{x+h}{tm^2}\right),\log\left(\frac{tm^2}{x}\right)\right)\ll h^{3/4}x^\epsilon. %
\end{equation}

We turn to the evaluation of the moment of order \(2\). 
Since $\eta>3/4+\alpha$, by \eqref{eq_kolm} and Lemma~\ref{lemma6}, we obtain for some positive constant \(C\),
\begin{equation}
\begin{aligned}
Ch 
& \le C'\sum_{x<n\le x+h} \lambda_{\f}(n)^2 
\\
& \le \sum_{m\le \sqrt{x+h}} \tau(m)^4
\sumsqf_{\frac{x}{m^2}\leq t\leq \frac{x+h}{m^2}} \lambda_{\f}(t)^2.
\end{aligned}
\end{equation}
Next we prove  that  \(\sqrt{x+h}\) can be replaced by some constant \(A\) in the outer sum up to the cost of a replacement of a smaller  \(C\). Indeed  we will prove, for any fixed $A>0$,  
\begin{equation}
\sum_{A< m\le \sqrt{x+x^\eta}} \tau(m)^4
\sumsqf_{\frac{x}{m^2}\leq t\leq \frac{x+x^\eta}{m^2}} \lambda_{\f}(t)^2 
\ll x^{\eta} A^{-1+\varepsilon}.
\end{equation}
Note that %
\begin{align}%
\sum_{\sqrt{x}\leq m\le \sqrt{x+x^\eta}} \tau(m)^4\sumsqf_{\frac{x}{m^2}\leq t\leq \frac{x+x^\eta}{m^2}} \lambda_{\f}(t)^2%
&=%
\sum_{\sqrt{x}\leq m\le \sqrt{x+x^\eta}} \tau(m)^4\sumsqf_{t\leq \frac{x+x^\eta}{m^2}} \lambda_{\f}(t)^2%
\notag\\
&\ll x^{1/2+\varepsilon} %
\label{eq_deraj}
\end{align}
by~\eqref{eq_kolm}. In light of~\eqref{eq_deraj}, \eqref{eq_kolm}  and \eqref{defvartheta}, it suffices to evaluate
\begin{equation}
\sum_{A<m\le \sqrt{x}} \tau(m)^4
\min\bigg\{\max\bigg[\frac{x^\eta}{m^2},  \bigg(\frac{x}{m^2}\bigg)^{3/4+\alpha+\varepsilon}\bigg],
\bigg(1+\frac{x^\eta}{m^2}\bigg)\frac{x^{2\alpha}}{m^{4\alpha}}\bigg\}.
\end{equation}
Write $y=x^{\eta}/m^2$ and $Y=x/ m^2$, then $0<y<Y $ and $Y\gg 1$.  Note $2\alpha<3/4+\alpha$. The term $\min\{\cdots\}$ in the preceding formula  is then handled by observing
\begin{align*}
\min\left\{\max(y,Y^{3/4+\alpha+\varepsilon}), (1+y)Y^{2\alpha}\right\} 
&\ll \left\{
\begin{array}{ll}
Y^{2\alpha} & \mbox{ if $y\le 1$},\\
yY^{2\alpha}  & \mbox{ if $1< y \le Y^{3/4-\alpha}$},\\
Y^{3/4+\alpha+\varepsilon} & \mbox{ if $Y^{3/4-\alpha}< y \le Y^{3/4+\alpha+\varepsilon}$},\\
y & \mbox{ if $Y^{3/4+\alpha+\varepsilon}<y<Y$}.
\end{array}
\right.
\end{align*}
We split the sum over $m$ into 4 subsums with the ranges of summation dividing at the points for which $y=1$, $y=Y^{3/4-\alpha}$ and $y=Y^{3/4+\alpha+\varepsilon}$ respectively. Write  
\begin{equation*}
\mbox{$\eta_3=\frac{\eta}{2}$ , \quad $\eta_2= \frac{\eta-3/4+\alpha}{1/2+2\alpha}$ , \quad $\eta_1= \frac{\eta-(3/4+\alpha+\varepsilon)} {1/2-2(\alpha+\varepsilon)}$ }
\end{equation*} 
(note $\eta_3>\eta_2>\eta_1>0$). The 4 subsums are evaluated via the following summations:
\begin{align*}
\sum_{x^{\eta_3}<m\le \sqrt{x}} 
\tau(m)^4 \frac{x^{2\alpha}}{m^{4\alpha}}
& \ll x^{2\alpha + (1-4\alpha)/2+\varepsilon} 
= x^{1/2+\varepsilon} =o(x^\eta),
\\
\sum_{x^{\eta_2}<m\le x^{\eta_3}} 
\tau(m)^4 \frac{x^{\eta+2\alpha}}{m^{2+4\alpha}} 
& \ll x^{\eta+2\alpha- \eta_2(4\alpha+1)+\varepsilon}= x^{\eta-\frac{(\eta-3/4)(1+4\alpha)}{1/2+2\alpha}+\varepsilon},
\\
\sum_{x^{\eta_1}<m\le x^{\eta_2}} 
\!\!\!\! \tau(m)^4 \bigg(\frac{x}{m^2}\bigg)^{3/4+\alpha+\varepsilon}
& \ll x^{3/4+\alpha- \eta_1(2\alpha+1/2)+\varepsilon}
= x^{\eta-\frac{\eta-(3/4+\alpha+\varepsilon)}{1/2-2(\alpha+\varepsilon)}+\varepsilon},
\\
\sum_{A<m\le x^{\eta_1}} \tau(m)^4 \frac{x^\eta }{m^2}
& \ll x^\eta A^{-1+\varepsilon}.
\end{align*}

By taking a large enough constant $A$, we infer that 
\[%
\sum_{m\leq A}\tau(m)^4\sumsqf_{\frac{x}{m^2}<t<\frac{x+h}{m^2}}\lambda_{\f}(t)^2
\ge \big(C-O(A^{-1+\varepsilon})\big)h
\gg h. %
\]
This equation remains true if we replace \((x,h)\) by \((x+h/4,h/2)\), so %
\begin{equation}\label{eq_JWJWJW}%
\sum_{m\leq A}\tau(m)^4\sumsqf_{\frac{x+h/2}{m^2}<t\leq\frac{x+3h/4}{m^2}}\lambda_{\f}(t)^2
\gg h. %
\end{equation}
Moreover %
\begin{multline}\label{eq_JW}%
\sum_{m\leq A}\tau(m)^4\sumsqf_{\frac{x}{m^2}<t<\frac{x+h}{m^2}}\lambda_{\f}(t)^2\min\left(\log\left(\frac{x+h}{tm^2}\right),\log\left(\frac{tm^2}{x}\right)\right)\\%
\geq\sum_{m\leq A}\tau(m)^4\sumsqf_{\frac{x+h/4}{m^2}<t\leq\frac{x+3h/4}{m^2}}\lambda_{\f}(t)^2\min\left(\log\left(\frac{x+h}{tm^2}\right),\log\left(\frac{tm^2}{x}\right)\right)
\end{multline}
and, if \(t\in\left[\frac{x+h/4}{m^2},\frac{x+3h/4}{m^2}\right]\) then %
\begin{equation}\label{eq_JWJW}%
\frac{x}{h}\min\left(\log\left(\frac{x+h}{tm^2}\right),\log\left(\frac{tm^2}{x}\right)\right)\gg 1. %
\end{equation}
We deduce from~\eqref{eq_JW}, \eqref{eq_JWJW} and \eqref{eq_JWJWJW} that %
\begin{equation}\label{eq_momentdeux}%
\sum_{m\leq A}\tau(m)^4\sumsqf_{\frac{x}{m^2}<t<\frac{x+h}{m^2}}\lambda_{\f}(t)^2\min\left(\log\left(\frac{x+h}{tm^2}\right),\log\left(\frac{tm^2}{x}\right)\right)\gg\frac{h^2}{x}. %
\end{equation}
This is our moment of order \(2\).
%.............................
\subsection{Implication on the number of sign changes}\label{thm2p2}
%.............................
We use~\eqref{eq_momentun} and~\eqref{eq_bound} to write %
\begin{align}
&\sum_{m\leq A}\sumsqf_{\frac{x}{m^2}<t<\frac{x+h}{m^2}}\left(\abs*{\lambda_{\f}(t)}+\epsilon_m\lambda_{\f}(t)\right)\min\left(\log\left(\frac{x+h}{tm^2}\right),\log\left(\frac{tm^2}{x}\right)\right)\\
&\phantom{xxx}\gg %
\sum_{m\leq A}\sumsqf_{\frac{x}{m^2}<t<\frac{x+h}{m^2}}t^{-\alpha}\lambda_{\f}(t)^2\min\left(\log\left(\frac{x+h}{tm^2}\right),\log\left(\frac{tm^2}{x}\right)\right)+O\left(h^{3/4+\epsilon}\right)\\
&\phantom{xxx}\gg x^{-1-\alpha}h^2+O\left(h^{3/4+\epsilon}\right)\label{eq_contradi} %
\end{align}
by~\eqref{eq_momentdeux}. If \(\eta>\frac{4}{5}(1+\alpha)\),  we deduce %
\[%
\sum_{m\leq A}\sumsqf_{\frac{x}{m^2}<t<\frac{x+h}{m^2}}\left(\abs*{\lambda_{\f}(t)}+\epsilon_m\lambda_{\f}(t)\right)\min\left(\log\left(\frac{x+h}{tm^2}\right),\log\left(\frac{tm^2}{x}\right)\right)\gg x^{2\eta-1-\alpha}. 
\]

Assume that, for all \(m\in\{1,\dotsc,A\}\), there exists \(\epsilon_m\in\{-1,1\}\) such that the sign of \(\lambda_{\f}(t)\) is \(-\epsilon_m\) for every squarefree \(t\in\left]\frac{x}{m^2},\frac{x+h}{m^2}\right[\). Then, %
\[%
\sum_{m\leq A}\sumsqf_{\frac{x}{m^2}<t<\frac{x+h}{m^2}}\left(\abs*{\lambda_{\f}(t)}+\epsilon_m\lambda_{\f}(t)\right)\min\left(\log\left(\frac{x+h}{tm^2}\right),\log\left(\frac{tm^2}{x}\right)\right)=0 %
\]
in contradiction with~\eqref{eq_contradi}. 
Consequently, there exists \(m\in\{1,\dotsc,A\}\) such that the interval \(\left]\frac{x}{m^2},\frac{x+h}{m^2}\right[\) contains squarefree integers \(t\) and \(t'\) satisfying %
\[%
\abs*{\lambda_{\f}(t)}=\lambda_{\f}(t)\neq 0
\qquad\text{and}\qquad
\abs*{\lambda_{\f}(t')}=-\lambda_{\f}(t')\neq 0
\]
\emph{i.e.} \(\lambda_{\f}(t)\lambda_{\f}(t')<0\). 

Let \(X\) be any sufficiently large number. Write \(B=(1+1/A)^2\), \(H=(BX)^\eta\) and \(J=\lfloor(B-1)X/H\rfloor\). For any \(j\in\{0,\dotsc,J-1\}\) and any \(m\in\{1,\dotsc,A\}\), let %
\[%
I_j(m)=\left]\frac{X+jH}{m^2},\frac{X+(j+1)H}{m^2}\right[.%
\]
The interval \(I_J(m+1)\) is on the left side of \(I_0(m)\). Moreover, if \(j\neq k\), then \(I_j(m)\cap I_k(m)=\emptyset\). It follows that the \(AJ\) intervals \(I_j(m)\) are disjoint. %
%\begin{itemize}
%\item if \(m_0<m_1\) then \(I_{j_0}(m_0)\subset\left]\frac{X}{m_0^2},\frac{X+JH}{m_0^2}\right[\), \(I_{j_1}(m_1)\subset\left]\frac{X}{m_1^2},\frac{X+JH}{m_1^2}\right[\) and \(\frac{X+JH}{m_1^2}<\frac{X}{m_0^2}\),
%\item if \(m_0=m_1\) and \(j_0<j_1\), then \(\frac{X+(j_0+1)H}{m_0^2}\leq\frac{X+j_1H}{m_0^2}\).
%\end{itemize}
Since, for any \(j\), there exists \(m\) such that \(I_j(m)\) contains a sign change, we obtain at least \(J\gg X^{1-\eta}\) sign changes over the interval \([1,X]\). The proof is complete after replacing \(\eta\) by \(\eta+\epsilon\).
%.............................
\bibliography{Lau_Royer_Wu_2014}
\bibliographystyle{amsplain}
%..............................
\end{document}